\documentclass[12pt,reqno]{amsart}

\usepackage{a4,amsmath,amssymb,url}
\usepackage{enumitem}

\usepackage{verbatim}

\usepackage{xcolor}

\theoremstyle{plain}
\newtheorem{thm}{Theorem}[section]
\newtheorem{lem}[thm]{Lemma}
\newtheorem{cor}[thm]{Corollary}

\theoremstyle{definition}

\theoremstyle{remark}
\newtheorem{rem}[thm]{Remark}

\newcommand{\N}{{\mathbb{N}}}
\newcommand{\Z}{{\mathbb{Z}}}
\renewcommand{\P}{\mathcal{P}}
\newcommand{\F}{\mathbb{F}}

\newcommand{\I}{I}

\newcommand{\st}{\: : \: }

\newcommand{\leqsd}{\leq_{\textup{sd}}}
\newcommand{\leqlex}{\leq_{\textup{sl}}}
\newcommand{\lelex}{<_{\textup{sl}}}

\newcommand{\kN}[2]{\,_{#1} #2}

\DeclareMathOperator{\tr}{tr}
\newcommand{\SL}[2]{\mathrm{SL}_{#1}(#2)}
\newcommand{\Aut}{\mathrm{Aut}}

\setenumerate{label=\textup{(\roman*)}, ref=\roman*, widest=(iii), leftmargin=10mm}

\date{\today}
\thanks{The first two authors were supported by the National Science Foundation under Grant No. DMS 1500254.}

\begin{document}


\title[Quotients of subdirect products of perfect groups]{Solvable quotients of subdirect products of perfect groups are nilpotent}

\author[K. Kearnes]{Keith Kearnes}
\author[P. Mayr]{Peter Mayr}
\address[K. Kearnes and P. Mayr]{Department of Mathematics, CU Boulder, USA}
\email{keith.kearnes@colorado.edu, peter.mayr@colorado.edu}
\author[N. Ru\v{s}kuc]{Nik Ru\v{s}kuc}
\address[N. Ru{\v{s}}kuc]{School of Mathematics and Statistics, University of St Andrews, St Andrews, Scotland, UK}
\email{nik.ruskuc@st-andrews.ac.uk}

\begin{abstract}
  We prove the statement in the title and
 exhibit examples of quotients of arbitrary nilpotency class. This answers a question by D.~F.~Holt.
\end{abstract}


\maketitle

\section{Introduction}

A subgroup $S$ of a direct product $G_1\times\dots\times G_n$ is said to be \emph{subdirect} if its projection to
each of the factors $G_i$ is surjective.
Subdirect products have recently been a focus of interest from combinatorial and algorithmic point of view; see,
for example, \cite{bridson}.
They are also frequently used in computational finite group theory, notably to  construct perfect groups;
see \cite{holtplesken}.

In September 2017 D.~F.~Holt asked on the group pub forum whether a solvable quotient of a subdirect product of
perfect groups is necessarily abelian. The question was re-posted on mathoverflow~\cite{holtqu} where Holt gave
some background:
``This problem arose in a study of the complexity of certain algorithms for finite permutation and matrix groups. The group $S$ in the applications is the kernel of the action of a transitive but imprimitive permutation (or matrix) group on a block system. So in that situation the $G_i$ are all isomorphic, and $\Aut(S)$ induces a transitive action on the direct factors $G_i$ ($1\leq i\leq n$).''

In the present note we start from the observation that every solvable quotient of a subdirect product of perfect
groups is nilpotent, proved in Section \ref{sec:nilp}, and then proceed to construct several examples demonstrating
that the nilpotency class can be arbitrarily high.
Specifically, in Section \ref{sec:infG} we exhibit an infinite perfect group $G$
such that for every $d$ there is a subdirect subgroup of $G^{2^d}$ with a nilpotent quotient of degree $d$.
In Section \ref{sec:finG1} we exhibit a perfect subgroup $G$ of the special linear group
$\SL{6}{\F_4}$ such that $G^4$ contains a subdirect subgroup with a nilpotent quotient of class $2$. 
This example is generalized in Section \ref{sec:finG2} to a perfect subgroup $G$ of the special linear group
$\SL{2n}{\F_4}$ such that $G^{2^{n-1}}$ contains a subdirect subgroup with a nilpotent quotient of class $n-1$.
These examples are all based on a particular subdirect product construction, which is described in Section \ref{sec:constr},
and whose lower central series is computed in Section \ref{sec:lcs}.

Throughout the paper we will use the following notation.
Let $\mathbb{N}$ denote the set of positive integers.
For $n\in\mathbb{N}$, we write $[n]$ to denote $\{1,\dots,n\}$.
For elements $x,y$ of a group $G$ we write
$x^y:=y^{-1}xy$ and $[x,y]:=x^{-1}y^{-1}xy$.
The 
 higher
commutators in a group $G$ will always  be assumed to be left associated, i.e.
\[
[x_1,\dots,x_k]:=\bigl[ [x_1,x_2],\dots,x_k\bigr],\ 
[N_1,\dots,N_k]:=\bigl[ [N_1,N_2],\dots,N_k\bigr]
\]
for $x_i\in G$ and $N_i\unlhd G$.
Furthermore, for $k\geq 1$ we let
\[
  \gamma_k(G):=[\underbrace{G,\dots,G}_k],
\ \gamma_1(G) := G,
\]
 denote the $k$th term in the lower central series of $G$ and, for $N\unlhd G$ and $k\geq 0$, write
\[
[G, \kN{k}{N}]:=[G,\underbrace{N,\dots,N}_k],\ 
[G,\kN{0}{N}] := G.
\]

The following two facts about commutators for a group $G$ with a normal subgroup $N$ will frequently be used
($k,l\in\N, m\in\N\cup\{0\}$):

\begin{enumerate}[label=\textsf{[C\arabic*]}, widest=(C2), leftmargin=10mm]
\item
\label{it:C1}
$[\gamma_k(N),\gamma_l(N)]\leq \gamma_{k+l}(N)$;
\item
\label{it:C2}
$\bigl[[G, \kN{m}{N}],\gamma_l(N)\bigr]\leq [G, \kN{m+l}{N}]$.
\end{enumerate}
The first can be found in
\cite[5.1.11(i)]{robinson}.
The second follows by induction on $l$.
The case $l = 1$ holds by the definition of $[G,\kN{k}{N}]$ and the fact that $\gamma_1(N)=N$.
 For $l>1$ we have  
\begin{align*} 
  \bigl[[G, \kN{m}{N}],\gamma_l(N)\bigr] & \leq \Bigl[\bigl[[G, \kN{m}{N}],N\bigl],\gamma_{l-1}(N)\Bigr] \cdot \Bigl[\bigl[[G, \kN{m}{N}],\gamma_{l-1}(N)\bigl],N\Bigr] \\
  & \mbox{}\hspace{5cm} \text{by the Three Subgroup Lemma} \\
                                        & = \bigl[[G, \kN{m+1}{N}],\gamma_{l-1}(N)\bigr] \cdot \Bigl[\bigl[[G, \kN{m}{N}],\gamma_{l-1}(N)\bigr],N\Bigr] \\
                                        & \leq [G, \kN{m+l}{N}] \cdot \bigl[[G, \kN{m+l-1}{N}],N\bigr] \text{  by induction} \\
& = [G, \kN{m+l}{N}].
\end{align*}

\section{Solvable quotients must be nilpotent}
\label{sec:nilp}

That solvable quotients of subdirect products of perfect groups are necessarily nilpotent 
is an easy consequence of the following more general observation.

\begin{thm} \label{thm1}
 Let $N\unlhd S\leqsd G_1\times \dots\times G_n$ such that $N,S$ are both subdirect products of groups $G_1,\dots,G_n$.
 Then $S/N$ is nilpotent of class at most $n-1$.
\end{thm} 
  
\begin{proof}
 For $i=1,\dots,n$, let $K_i\unlhd S$ be the kernel of the projection of $S$ onto $G_i$. 
Notice that $\bigcap_{i=1}^n K_i=1$.
 Moreover we claim that
\begin{equation} \label{eq:KiN}
  K_iN = S  \text{ for all } i=1,\dots,n.
\end{equation}
 For the inclusion $\geq$, let $s\in S$. Since the projection of $N$ onto $G_i$ is onto by assumption,
 we have $b\in N$ such that $b_i = s_i$. Then $sb^{-1}$ is in $K_i$ which proves~\eqref{eq:KiN}. 
 It now follows that
\[
  \gamma_n(S)=[K_1N,\dots,K_nN]\leq[K_1,\dots,K_n] N.
\]
 Since $[K_1,\dots,K_n]\leq \bigcap_{i=1}^n K_i=1$,
 we obtain $\gamma_n(S)\leq N$.
\end{proof}

\begin{rem}
 The analogue of Theorem~\ref{thm1} holds, with exactly the same proof, in any congruence modular variety when
 the normal subgroup $N$ is replaced  by a congruence $\nu$ of the subdirect product $S$ such that the join of
 $\nu$ with any projection kernel $\rho_i$ ($i\leq n$) yields the total congruence on $S$.
 In fact, then $S/\nu$ is supernilpotent of class at most $n-1$.
We refer the reader to \cite{FM:CTC} for the definition of commutators of congruences and to \cite{Mo:HCT} for supernilpotence and higher commutators in this general setting.
It then follows that the analogue of Theorem~\ref{thm1} holds for rings, $K$-algebras, Lie algebras, and loops.
\end{rem}

\begin{cor} \label{cor1}
Let $S\leqsd G_1\times \dots\times G_n$ be a subdirect product of perfect groups $G_1,\dots,G_n$,
and let $N\unlhd S$ be a normal subgroup of $S$.
If $S/N$ is solvable, then $S/N$ is nilpotent of class at most $n-1$.
\end{cor}

\begin{proof}
 For $i=1,\dots,n$, let $N_i$ denote the projection of $N$ on $G_i$. 
 Since each $G_i$ is perfect and $G_i/N_i$ is solvable by assumption, we obtain $G_i = N_i$ and the result
 follows from Theorem~\ref{thm1}.
\end{proof}

\section{A subdirect construction}
\label{sec:constr}

By Corollary~\ref{cor1}, every solvable quotient of the subdirect product of two perfect groups is abelian.
In what follows, we are going to show that in fact there exist subdirect powers of a perfect group $G$ with
quotients of arbitrarily large nilpotency class.
To this end we introduce a construction that is based on higher commutator relations from universal algebra~\cite{Mo:HCT}.

Let $G$ be a group, and let $d\in\N$. We will be working in the direct product of $2^d$ copies of $G$
with components indexed by the power set $\P:=\P([d])$.
The set $\P$ will be linearly ordered by the short-lex ordering;
specifically
\[
A\lelex B \text{ iff } |A|<|B| \text{ or } \bigl(|A|=|B| \text{ and }
\min(A\setminus B)<\min(B\setminus A)\bigr).
\]

Now, for any $A\subseteq [d]$,  $u\in G$ and $K\subseteq G$, define $\Delta_A(u)\in G^\P$ by
\[
\Delta_A(u)(B):=\left\{ \begin{array}{ll} u & \text{if } A\subseteq B\\
                                                           1 & \text{otherwise.}
                                  \end{array}\right.
\]
and
\[
\Delta_A(K):=\bigl\{ \Delta_A(u)\st u\in K\bigr\}.
\]

The following is an immediate consequence of the above definition:

\begin{lem}
\label{lem:DC}
If $K,L\unlhd G$ and $A,B\subseteq [d]$, then
\[
\bigl[\Delta_A(K),\Delta_B(L)\bigr]=\Delta_{A\cup B} \bigl([K,L]\bigr).
\]
\end{lem}

Now, for a normal subgroup $N\unlhd G$ and for $k\in [d+1]$ we define the following subset of $G^\P$:
\begin{equation}
\label{eq:S}
\Gamma_k:=\Gamma_k(G,N):= 
\prod_{|A|<k} \Delta_A\bigl( [G,\kN{|A|}{N}]\bigr)\cdot
\prod_{|A|\geq k} \Delta_A\bigl( \gamma_{|A|}(N)\bigr).
\end{equation}
 Here the factors in the products are ordered by short-lex on $\P$.
  
\begin{lem} \label{lem:KL}
 Let $A,B\subseteq [d]$. Then
\begin{enumerate}
\item $\bigl[ [G,\kN{|A|}{N}],[G,\kN{|B|}{N}] \bigr] \leq [G,\kN{|A\cup B|}{N}]$, 
\item $\bigl[ [G,\kN{|A|}{N}], \gamma_{|B|}(N) \bigr] \leq [G,\kN{|A\cup B|}{N}]$ for $B\neq\emptyset$ and
\item $\bigl[ \gamma_{|A|}(N), \gamma_{|B|}(N) \bigr] \leq \gamma_{|A\cup B|}(N)$ for $A,B\neq\emptyset$.
\end{enumerate}
\end{lem}
   
\begin{proof}
 Immediate from~\ref{it:C1} and~\ref{it:C2}.
\end{proof}

\begin{lem}
\label{lem:subgroup}
 $\Gamma_k(G,N)$ is a subdirect subgroup of $G^\P$ for each $k\in [d+1]$ and $N\unlhd G$.
\end{lem}

\begin{proof}
First notice that all factors in the definition of $\Gamma_k$ are subgroups. That $\Gamma_k$ projects onto each coordinate follows from the presence of 
$\Delta_\emptyset\bigl([G,\kN{0}{N}]\bigr)=\Delta_\emptyset(G)$, the diagonal of $G^\P$.

To prove that $\Gamma_k$ is a subgroup, consider two generic factors
$\Delta_A(K)$ and $\Delta_B(L)$ in its definition for $A\lelex B$.
That is, $K=[G,\kN{|A|}{N}]$ if $|A|<k$ and $K=\gamma_{|A|}(N)$ otherwise, and likewise for 
$L$. 
 By Lemma \ref{lem:DC}, we obtain
\begin{eqnarray*}
\Delta_B(L)\cdot\Delta_A(K)&=&\Delta_A(K)\cdot\Delta_B(L)\cdot\bigl[\Delta_B(L),\Delta_A(K)\bigr]\\
&=&\Delta_A(K)\cdot\Delta_B(L)\cdot\Delta_{A\cup B}\bigl( [K,L]\bigr).
\end{eqnarray*}
By Lemma~\ref{lem:KL}, we have $[K,L] \leq [G,\kN{|A\cup B|}{N}]$ if $|A\cup B|<k$; further
$[K,L] \leq \gamma_{|A\cup B|}(N)$ if $|A\cup B|\geq k$. Hence $\Delta_{A\cup B}\bigl( [K,L]\bigr)$
 is contained in another factor of the product defining $\Gamma_k$. Since $B\leqlex A\cup B$, the result
 follows.
\end{proof}

\begin{rem}
 In reference to Holt's motivation, we note that the automorphism group of $S:=\Gamma_1$ acts transitively
 on the index set $\P$.
 More precisely, for $i\in [d]$, the permutation on $\P$ that
 maps a subset $A$ of $[d]$ to its symmetric difference with $\{i\}$ induces an automorphism $r_i$ on $S$.
 This can be easily seen by the action of $r_i$ on the generators $\Delta_A(u)$ of $S$:
 $r_i(\Delta_A(u)) = \Delta_A(u)$ if $i\not\in A$; else
 $r_i(\Delta_A(u)) = \Delta_{A\setminus\{i\}}(u)\cdot\Delta_A(u)^{-1}$.
 Then $\langle r_1,\dots, r_d\rangle \leq \Aut(S)$ is isomorphic to $\Z_2^d$ and acts regularly on the set of
 direct factors of $S$. 
\end{rem}

\section{The lower central series}
\label{sec:lcs}

Continuing with the construction from the previous section, we now let $S:=\Gamma_1$ and determine its
lower central series:

\begin{lem}
\label{lem:lcc}
If $G'=G$, then $\gamma_k(S)=\Gamma_k$ for all $k\in [d+1]$.
\end{lem}

\begin{proof}
 We need to show that
\begin{equation}
\label{eq:Gamma}
[\Gamma_k,\Gamma_1]=\Gamma_{k+1}\ \text{ for } k\in[d].
\end{equation}

For the inclusion $\leq$ it is sufficient to prove that
\begin{equation}
\label{eq:DAX}
\bigl[ \Delta_A(K),\Delta_B(L)\bigr] = \Delta_{A\cup B}\bigl([K,L]\bigr)\leq \Gamma_{k+1},
\end{equation}
where
\[
 K=\left\{ \begin{array}{ll} [G, \kN{|A|}{N}] & \text{if } |A|<k, \\ \gamma_{|A|}(N) & \text{otherwise}, \end{array}\right.\ \ L=\left\{ \begin{array}{ll} G & \text{if } B=\emptyset, \\ \gamma_{|B|}(N) & \text{otherwise}. \end{array}\right.
\]
 Note that~\eqref{eq:DAX} follows from the claim that
\begin{equation}
\label{eq:KL}
[K,L]\leq \begin{cases} [G, \kN{|A\cup B|}{N}] & \text{if } |A\cup B|<k+1, \\
    \gamma_{|A\cup B|}(N) & \text{otherwise}.
 \end{cases}
\end{equation}
 The assertion for $|A\cup B| < k+1$ follows from Lemma~\ref{lem:KL} except when $|A|=k$ and $B\subseteq A$.
 In this case we use~\ref{it:C2} to obtain
\[ [K,L] = [\gamma_{|A|}(N),L] \leq [G, \gamma_{|A|}(N)] \leq [G, \kN{|A|}{N}] = [G, \kN{|A\cup B|}{N}]. \]
Claim~\eqref{eq:KL} for $|A\cup B|\geq k+1$ is also immediate from Lemma~\ref{lem:KL} and the observation that
 $[G, \kN{|A\cup B|}{N}] \leq \gamma_{|A\cup B|}(N)$ for $A\cup B\neq\emptyset$.
 Hence~\eqref{eq:KL},~\eqref{eq:DAX}, and the inclusion $\subseteq$ of~\eqref{eq:Gamma} are proved.

For the inclusion $\supseteq$ of~\eqref{eq:Gamma} we show that
\begin{equation}
\label{eq:DAX2}
\Delta_A(K)\subseteq [\Gamma_k,\Gamma_1],
\end{equation}
where
\[
K=\left\{ \begin{array}{ll} [G,\kN{|A|}{N}] & \text{if } |A|<k+1 \\
                                        \gamma_{|A|}(N) & \text{if } |A|\geq k+1.
               \end{array}\right.
\]
Now, if $A=\emptyset$, the assumption that $G=G'$ yields
\[
\Delta_\emptyset(K)=\Delta_\emptyset(G)=\Delta_\emptyset\bigl([G,G]\bigr)
=\bigl[ \Delta_\emptyset(G),\Delta_\emptyset(G)\bigr]\subseteq [\Gamma_k,\Gamma_1].
\]
Assume now that $A\neq\emptyset$, and write it as $A=B\cup C$, where $|B|=|A|-1$ and $|C|=1$.
If $|A|<k+1$, then $|B|<k$ and by Lemma~\ref{lem:DC}
\[
\Delta_A(K)=\Delta_{B\cup C}\bigl( [G,\kN{|B|+| C|}{N}]\bigr)=
\Bigl[ \Delta_B\bigl( [G,\kN{|B|}{N}]\bigr),\Delta_C\bigl(N\bigr)\Bigr]
\subseteq [\Gamma_k,\Gamma_1].
\]
Likewise, if $|A|\geq k+1$, then $|B|\geq k$ and
\[
\Delta_A(K)=\Delta_{B\cup C}\bigl( \gamma_{|B|+| C|}(N)\bigr)
=\Bigl[ \Delta_B\bigl( \gamma_{|B|}(N)\bigr),\Delta_C\bigl(N\bigr)\Bigr]
\subseteq [\Gamma_k,\Gamma_1].
\]
This completes the proof of \eqref{eq:DAX2}, the inclusion $\supseteq$ of~\eqref{eq:Gamma}
 and the lemma.
\end{proof}

We can now prove our main result:

\begin{thm}
\label{thm2}
Let $G$ be a perfect group, $d\in\N$, and $N$ a normal subgroup such that $\gamma_d(N)>[G,\kN{d}{N}]$.
Then $S=\Gamma_1(G,N)$ is a subdirect subgroup of $G^{2^d}$ which has a nilpotent quotient of class $d$.
\end{thm}

\begin{proof}
 That $S$ is subdirect was proved in Lemma \ref{lem:subgroup}.
By Lemma \ref{lem:lcc} and \eqref{eq:S}, we have 
\begin{align*}
 &\gamma_d(S) = \Gamma_d=\prod_{|A|<d} \Delta_A\bigl( [G,\kN{|A|}{N}]\bigr)\cdot \Delta_{[d]}\bigl( \gamma_{d}(N)\bigr), \\
 &\gamma_{d+1}(S) = \Gamma_{d+1}=\prod_{A\in\P} \Delta_A\bigl( [G,\kN{|A|}{N}]\bigr).
\end{align*}
Hence the kernels of the projections  on $\P\setminus \{[d]\}$ of $\gamma_d(S)$ and $\gamma_{d+1}(S)$, respectively, are
 $\Delta_{[d]}\bigl( \gamma_{d}(N)\bigr)$ and $ \Delta_{[d]}\bigl( [G,\kN{d}{N}]\bigr)$, respectively.
 Since the assumption $\gamma_d(N)>[G,\kN{d}{N}]$ yields
 $$\Delta_{[d]}\bigl( \gamma_{d}(N)\bigr) > \Delta_{[d]}\bigl( [G,\kN{d}{N}]\bigr),$$
 we obtain $\gamma_d(S)>\gamma_{d+1}(S)$. Thus $S/\gamma_{d+1}(S)$ is indeed $d$-nilpotent but not ($d-1$)-nilpotent.
\end{proof}

\begin{rem}
 The 2018 version of the GAP library `Perfect Groups' by Felsch, Holt and Plesken \cite{GAP4,GAPPerfect} provides a list of all perfect groups whose sizes are less than $10^6$
 excluding $11$ sizes.
Unfortunately none of the groups $G$ in this library have a normal subgroup $N$
 with $\gamma_d(N)>[G,\kN{d}{N}]$ for $d>1$. Hence Theorem~\ref{thm2} does not yield non-abelian nilpotent quotients
 of $\Gamma_1(G,N)$ for them.
\end{rem}

\section{An infinite example}
\label{sec:infG}

In this section we demonstrate the existence of a perfect group $G$ satisfying the assumptions of
Theorem~\ref{thm2}. More precisely, we will construct an infinite group $G$ with a normal subgroup $N$ such
that
\begin{equation}
\label{eq:ex1}
\gamma_d(N)>[G,\kN{d}{N}] \text{ for all } d\in\N.
\end{equation}

Let $X$ be a countably infinite alphabet with elements
\[
X:=\bigl\{ x_{i,w}\st i\in\N,\ w\in\{0,1\}^\ast\bigr\}
\cup \bigl\{ y_i\st i\in\N\bigr\}.
\]
 Here $\{0,1\}^\ast$ denotes the set of all words over $\{0,1\}$ including the empty word $\epsilon$.
Let $F$ be the free group on $X$.
The derived subgroup $F^\prime$ is free by Nielsen's Theorem, and it is routine to see
that it has a free basis $Z$ which contains all the commutators $[u,v]$ with $u,v\in X$, $u\neq v$.
For $i\in\N$, $w\in\{0,1\}^\ast$ set
\[
h(x_{i,w}):=[x_{i,w0},x_{i,w1}]\in Z,
\]
and extend $h$ to a bijection $h:X\rightarrow Z$, and then to an isomorphism $h:F\rightarrow F^\prime$.

Now let $F_i:=F$ for $i\in\N$, and let $G$ be the limit of the directed system
\[
F_1 \xrightarrow{h} F_2 \xrightarrow{h} F_3\xrightarrow{h}\dots
\]
Thus, there exist embeddings $h_i : F_i\rightarrow G$ ($i\in\N$) such that 
$h_i=h_{i+1}h$, $h_1(F_1) \leq h_2(F_2) \leq\dots$ and $G=\bigcup_{i\in\N} h_i(F_i)$.

The group $G$ is perfect, since $h_i(F_i)=h_{i+1}(F_{i+1})^\prime$.
Also, from $h_1(F_1)=h_i(F_i)^{(i-1)}$, it follows that $N:=h_1(F_1)\unlhd G$.
We claim that
\eqref{eq:ex1} holds for this $G$ and $N$.

Suppose to the contrary that $\gamma_d(N)=[G,\kN{d}{N}]$ for some $d$.
Consider the element
\[
u:= [x_{1,\epsilon},x_{2,\epsilon},\dots,x_{d,\epsilon}]\in \gamma_d(F_1).
\]
Then $h_1(u)\in\gamma_d(N)$ and hence $h_1(u)\in [G,\kN{d}{N}]$.
It follows that
\[
h_1(u)=v_1v_2\dots v_k,
\]
where
\[
v_i=c_i^{-1}[g_i,n_{i1},\dots,n_{id}]^{e_i}c_i
\]
for $e_i\in\{-1,1\}$, $c_i,g_i\in G$, $n_{ij}\in N$ for all $i\leq k, j\leq d$.
From $G=\bigcup_{t\in\N}h_t(F_t)$ it follows that there exists $t\in\N$ such that
$c_i,g_i\in h_t(F_t)$ for all $i=1,\dots,k$.
Therefore we have
\[
h_1(u)\in \bigl[h_t(F_t),\kN{d}{h_1(F_1)}\bigr].
\]
Since $h_1=h_th^{t-1}$ and $h_t:F_t\rightarrow G$ is an injection, this implies
\[
h^{t-1}(u)\in \bigl[ F_t,\kN{d}{h^{t-1}(F_1)}\bigr].
\]
Recall $F_1=F_t=F$ and $h^{t-1}(F)=F^{(t-1)}$ to obtain
\begin{equation} \label{eq:ht}
h^{t-1}(u)\in [F,\kN{d}{F^{(t-1)}}].
\end{equation}
Note that
\[
h^{t-1}(x_{i,\epsilon})=h^{t-2}\bigl( [x_{i,0},x_{i,1}]\bigr)=
h^{t-3}\Bigl( \bigl[ [ x_{i,00},x_{i,01}],[x_{i,10},x_{i,11}]\bigr]\Bigr)=\dots
\]
generates $F^{(t-1)}$ as a fully invariant subgroup of $F$. Hence
\[
h^{t-1}(u)=\bigl[ h^{t-1}(x_{1,\epsilon}),\dots, h^{t-1}(x_{d,\epsilon})\bigr]
\]
 generates $\gamma_d(F^{(t-1)})$ as a fully invariant subgroup.
 Since $[F,\kN{d}{F^{(t-1)}}]$ is also fully invariant,~\eqref{eq:ht} implies the inclusion $\leq$ in
\[
\gamma_d(F^{(t-1)})= [F,\kN{d}{F^{(t-1)}}].
\]
 The converse inclusion $\geq$ is trivial.
Since $F$ is free of countable rank, this implies that
\[
\gamma_d(H^{(t-1)})=[H,\kN{d}{H^{(t-1)}}]
\]
for \emph{all} countable (not necessarily perfect) groups $H$.
But this is false, one counter-example being the group of upper unitriangular $n\times n$ matrices over any field for sufficiently large $n$.
This contradiction establishes Claim \eqref{eq:ex1}.

\section{A finite example}
\label{sec:finG1}

In this section we exhibit a finite group $G$ with a normal subgroup $N$ such that
$[N,N]>[G,N,N]$.

Let $\F:=\F_4$ be the field of order $4$, and let $\alpha\in \F$ be a primitive element; recall that $\alpha^2=\alpha+1$.
Recall that the special linear group $S:=\SL{2}{\F}$ is perfect~\cite[3.2.8]{robinson};
in fact, 
$S$ is isomorphic to the alternating group $A_5$.
 Let us denote by $\I$ the identity of $S$.
Also let
\[
T:=\{ A\in \F^{2\times 2}\st \tr(A)=0\},
\] 
the set of all matrices of trace $0$.

In what follows we will work in the special linear group $\SL{6}{\F}$.
Its elements will be considered and written as $3\times 3$ block matrices over $\F^{2\times 2}$.
First define the following three subgroups of $\SL{6}{\F}$,
where the omitted entries are understood to be $0$-blocks:
\begin{gather*}
H:= \left\{ 
    \begin{pmatrix} A & & \\ &A& \\ &&A \end{pmatrix}
    \st A\in S\right\} ,
\\
M:= \left\{ 
    \begin{pmatrix} \I & B & D \\ &\I& C \\ &&\I \end{pmatrix}
    \st B,C,D\in\F^{2\times 2},\ \tr(B)=\tr(C)=0\right\} ,
\\
N:= \left\{ 
    \begin{pmatrix} \I & b\I & D \\ &\I& c\I \\ &&\I \end{pmatrix}
    \st b,c\in\F,\ D\in\F^{2\times 2}\right\} ,
\end{gather*}
and then set
\[
G:=HM.
\]

\begin{lem}
\label{lem:SL6}
The following hold for $G$, $N$ and $M$ as defined above:
\begin{enumerate}
\item   \label{it:GN1}
$M\unlhd G\leq \SL{6}{\F}$.
\item   \label{it:GN2}
$G$ is perfect.
\item   \label{it:GN3}
$N\unlhd G$.
\item   \label{it:GN4}
$[N,N]>[G,N,N]$.
\end{enumerate}
\end{lem}

We note that $G$ has size $60\cdot 4^{10} = 62\, 914\, 560$ which exceeds the maximal size of perfect groups
 listed in the `Perfect Groups' GAP library~\cite{GAPPerfect}.

In the proof of the above results, we will make use of the following two technical observations.
Throughout what follows we will use the standard exponentiation notation for matrix conjugation: $B^A:=A^{-1}BA$.

\begin{lem}
\label{lemA}
The set $T$, considered as an $\F$-module, is spanned by the set
\[
U:=\{ -B+B^A\st B\in T,\ A\in S\}.
\]
\end{lem}

\begin{proof}
For every $x\in\F$ we have
\[
\begin{pmatrix}   x & x^2 \\ 0 & x  \end{pmatrix} =
\begin{pmatrix}   0 & 0 \\ 1 & 0  \end{pmatrix} +
\begin{pmatrix}   1 & x \\ 0 & 1  \end{pmatrix}
\begin{pmatrix}   0 & 0 \\ 1 & 0  \end{pmatrix}
\begin{pmatrix}   1 & x \\ 0 & 1  \end{pmatrix}
\in U,
\]
and hence
\[
\begin{pmatrix}   0 & 1 \\ 0 & 0  \end{pmatrix} =
\alpha \begin{pmatrix}   1 & 1 \\ 0 & 1  \end{pmatrix} +
\begin{pmatrix}   \alpha & \alpha^2 \\ 0 & \alpha  \end{pmatrix}
\in\F U.
\]
By symmetry 
$\bigl(
\begin{smallmatrix} 0&0\\ 1&0\end{smallmatrix}\bigr)\in\F U$,
and also
$\I=      
\bigl(\begin{smallmatrix} 1&1\\ 0&1\end{smallmatrix}\bigr) +
\bigl(\begin{smallmatrix} 0&1\\ 0&0\end{smallmatrix}\bigr)
\in \F U$.
Since $T$ is spanned by
$\I$, 
$\bigl(\begin{smallmatrix} 0&1\\ 0&0\end{smallmatrix}\bigr)$ and
$\bigl(\begin{smallmatrix} 0&0\\ 1&0\end{smallmatrix}\bigr)$, the assertion follows.
\end{proof}

\begin{lem}
\label{lemB}
The $\F$-module spanned by the set
\[
U:=\{ BC\st B,C\in T \}
\]
is the entire $\F^{2\times 2}$.
\end{lem}

\begin{proof}
 The subset
\[
\left\{
\begin{pmatrix} 0&1\\ 0&0\end{pmatrix},
\begin{pmatrix} 0&0\\ 1&0\end{pmatrix}
\right\}
\cdot
\left\{
\I,
\begin{pmatrix} 0&1\\ 0&0\end{pmatrix},
\begin{pmatrix} 0&0\\ 1&0\end{pmatrix}
\right\}
\]
 of $U$ contains the natural basis of $\F^{2\times 2}$.
\end{proof}

\begin{proof}[Proof of Lemma \ref{lem:SL6}]
\eqref{it:GN1}
For
\begin{equation}
\label{eq14}
X:=\begin{pmatrix} A&&\\ &A&\\ &&A\end{pmatrix}\in H \text{ and } 
Y:=\begin{pmatrix} \I&B&D\\ &\I&C\\ &&\I\end{pmatrix}\in M
\end{equation}
we have
\[
  Y^X=
\begin{pmatrix} \I&B^A&D^A\\ &\I&C^A\\ &&\I\end{pmatrix} .
\]
From $\tr(B^A)=\tr(B)=0$ and $\tr(C^A)=\tr(C)=0$ it follows that
$Y^X\in M$.
Thus, $H$ normalizes $M$, and both assertions follow.
\medskip

\eqref{it:GN2}
Since $H\cong S$ is perfect, we have
\[
[G,G]=[HM,HM]=[H,H][H,M][M,M]=H[H,M]M'.
\]
So it suffices to show that
\begin{equation}
\label{eq15a}
[H,M]=M.
\end{equation}
Now, for $X\in H$, $Y\in M$ as in \eqref{eq14}, we have
\[ Y^{-1} = \begin{pmatrix} \I&-B&BC-D\\ &\I&-C\\ &&\I\end{pmatrix} \]
 and hence
\begin{equation}
\label{eq3}
[Y,X]=Y^{-1}Y^X= 
\begin{pmatrix} \I&-B+B^A& D^A-BC^A+BC-D \\ &\I&-C+C^A\\ &&\I\end{pmatrix} .
\end{equation}
Using Lemma \ref{lemA} and \eqref{eq3} twice, first with $C=D=0$, and then with $B=C=0$, we see that
\[
\begin{pmatrix} \I & P & 0 \\ &\I & 0\\ &&\I  \end{pmatrix},
\begin{pmatrix} \I & 0 & 0 \\ &\I & Q\\ &&\I  \end{pmatrix}
\in [H,M] 
\text{ for all } P,Q\in T.
\]
Multiplying these two matrices in both orders yields
\begin{equation}
  \label{eq12}
\begin{pmatrix} \I & P & 0 \\ &\I & Q\\ &&\I  \end{pmatrix},
\begin{pmatrix} \I &  P & PQ \\ &\I & Q\\ &&\I  \end{pmatrix}
\in [H,M] 
\text{ for all } P,Q\in T.
\end{equation}
Hence also
\begin{equation}
\label{eq12a}
\begin{pmatrix} \I & 0 & PQ \\ &\I & 0\\ &&\I  \end{pmatrix} =
\begin{pmatrix} \I & P & 0 \\ &\I & Q\\ &&\I  \end{pmatrix}^{-1}
\begin{pmatrix} \I & P & PQ \\ &\I & Q\\ &&\I  \end{pmatrix} \in [H,M].
\end{equation}
 Using Lemma \ref{lemB}, it follows that
\begin{equation}
\label{eq13}
 \begin{pmatrix} \I & 0 & R \\ &\I & 0\\ &&\I  \end{pmatrix}\in [H,M]
 \text{ for all } R\in\F^{2\times 2}.
\end{equation}
Multiplying the first matrix from \eqref{eq12} together with \eqref{eq13} yields
\[ \begin{pmatrix} \I & P & R \\ &\I & Q\\ &&\I  \end{pmatrix} =
\begin{pmatrix} \I & P & 0 \\ &\I & Q\\ &&\I  \end{pmatrix}
\begin{pmatrix} \I & 0 & R \\ &\I & 0\\ &&\I  \end{pmatrix}
\in [H,M] \] 
 for all $P,Q\in T,\ R\in \F^{2\times 2}$, proving Claim \eqref{eq15a} and part \eqref{it:GN2} of the lemma.

\eqref{it:GN3}
Again, let $X\in H$ and $Y\in M$ be as in \eqref{eq14}, and also let
\[
Z:=
\begin{pmatrix}  \I & b\I & E \\ &\I & c\I\\ &&\I  \end{pmatrix} \in N.
\]
Then
\[ Z^X=
  \begin{pmatrix}   \I & b\I & E^A \\ &\I & c\I \\ &&\I  \end{pmatrix} \text{ and }
Z^Y=
\begin{pmatrix}   \I & b\I & bC-cB+E \\ &\I & c\I \\ &&\I  \end{pmatrix}
\]
 are both in $N$.
Hence $N\unlhd G$.

\eqref{it:GN4}
To compute $[G,N]$, note that, for $X\in H$, $Y\in M$, $Z\in N$ as above, we have:
\[ [Z,X]=Z^{-1} Z^X=
\begin{pmatrix}   \I & 0 & -E+E^A \\ &\I & 0 \\ &&\I  \end{pmatrix}, \]
\begin{equation}
\label{eq22}
[Z,Y]=Z^{-1} Z^Y=
\begin{pmatrix}   \I & 0 & bC-cB \\ &\I & 0 \\ &&\I  \end{pmatrix}.
\end{equation}
Recalling Lemma \ref{lemA} it follows that
\begin{equation}
\label{eq23}
[G,N]=\left\{
\begin{pmatrix}   \I & 0 & P \\ &\I & 0 \\ &&\I  \end{pmatrix} \st P\in T\right\}.
\end{equation}
From  here it is immediate that $[G,N,N]$ is trivial.
 On the other hand, for $Y\in N$ with $B=b'\I$, $C=c'\I$ in~\eqref{eq22}
 we obtain
\[
[Z,Y]=
\begin{pmatrix}   \I & 0 & (bc'-b'c)\I \\ &\I & 0 \\ &&\I  \end{pmatrix}.
\]
 Hence
\[
[N,N]=
\left\{ \begin{pmatrix}   \I & 0 & d\I \\ &\I & 0 \\ &&\I  \end{pmatrix} \st d\in \F\right\}
\]
 is strictly bigger than $[G,N,N]$.
\end{proof}

\section{A family of finite examples}
\label{sec:finG2}

The construction from the previous section can be readily adapted to yield, for every $d\in\N$, a finite perfect group $G$ and a normal subgroup $N$ such that $\gamma_d(N)>[G,\kN{d}{N}]$.
We just give an outline.

Let $n\in\N$ be arbitrary. The field $\F$, group $S$ and set $T$ will be as in the previous section.
This time we will work in the group $\SL{2n}{\F}$, which we will treat as consisting of $n\times n$ block-matrices with $2\times 2$ entries.

For a matrix $A\in S$ we will write $\Delta(A)$ to denote the block diagonal matrix having $n$ copies of $A$ along the main diagonal. 
Also, for arbitrary $A\in \F^{2\times 2}$ and $i,j\in \{ 1,\dots,n\}$, $i\neq j$, we will use $E_{ij}(A)$ to denote the matrix which has the identity blocks on the main diagonal, has $(i,j)$ block equal to $A$, and has zero blocks elsewhere.

The basic ingredients for our construction are the following subgroups of $\SL{2n}{\F}$. Let
\[
H:= \{ \Delta(A)\st A\in S\} \cong S.
\]
 The second subgroup $M$ consists of all block upper unitriangular matrices with arbitrary
 matrices of trace $0$ on the super-diagonal, i.e., 
\[
M:=\bigl\{ (A_{ij}) \st A_{ii}=\I,\ A_{i,i+1}\in T,\ A_{ij}\in \F^{2\times 2}\ (j-i>1),\ 
A_{ij}=0\ (i<j)\bigr\}.
\]
The third subgroup $N$ consists of those elements of $M$ for which all the blocks on the super-diagonal belong to
the center of $S$, i.e,
\[
 N := \bigl\{ (A_{ij})\in M \st A_{i,i+1}=a_i\I\ (a_i \in\F) \}.
\]  
Finally, we let
\[
G:=HM.
\] 
  
\begin{lem}
\label{lem:SL2n}
The following hold for $G$, $N$ and $M$ as defined above:
\begin{enumerate}
\item   \label{it:G1}
$M\unlhd G\leq \SL{2n}{\F}$.
\item   \label{it:G2}
$G$ is perfect.
\item   \label{it:G3}
$N\unlhd G$.
\item   \label{it:G4}
$\gamma_d(N)>[G,\kN{d}{N}]$ for $d<n$.  
\end{enumerate}
\end{lem}

\begin{proof}
 Items~\eqref{it:G1} and~\eqref{it:G3} follow exactly as for Lemma~\ref{lem:SL6}.
 For proving~\eqref{it:G2}, we again just need 
\[
[H,M]=M.
\]
 For this, let $A\in S$, $B\in T$ and $E_{i,i+1}(B)\in M$ for $1\leq i\leq n-1$, and note that
\[
\bigl[ E_{i,i+1}(B),\Delta(A)\bigr]=E_{i,i+1}(-B+B^A).
\]
By Lemma \ref{lemA} it follows that
\begin{equation}
\label{eq31}
  \bigl\{ E_{i,i+1}(P)\st P\in T,\ 1\leq i\leq n-1\bigr\}\subseteq [H,M].
\end{equation}
 As in~\eqref{eq12a} we obtain
\[
\bigl[ E_{i,i+1}(P),E_{i+1,i+2}(Q)\bigr]=E_{i,i+2}(PQ).
\]
By Lemma \ref{lemB} it follows that
\[
\bigl\{ E_{i,i+2}(R)\st R\in\F^{2\times 2},\ 1\leq i\leq n-2\bigr\}\subseteq [H,M].
\]
A straightforward inductive argument now shows that
\begin{equation}
  \label{eq32}
  \bigl\{ E_{ij}(R)\st R\in\F^{2\times 2},\ j-i>1\bigr\}\subseteq [H,M].
\end{equation}
From \eqref{eq31} and \eqref{eq32} it easily follows that $[H,M]=M$, as claimed.

For~\eqref{it:G4}, we obtain as in~\eqref{eq23}
\[
  [G,N] = \bigl\{ (A_{ij}) \in N \st
   A_{i,i+1}=0,\ 
   A_{i,i+2}\in T \bigr\}.
\]
Assuming $n>d$, from this it is easy to see that
\[
  [G,\kN{d}{N}]= \bigl\{ (A_{ij}) \st A_{ij}=0\ (0<j-i\leq d),\ A_{i,i+d+1}\in T \bigr\}.
\]  
Finally,
\[
  \gamma_d(N)= \bigl\{ (A_{ij})\in N \st  A_{ij}=0\ (0<j-i\leq d-1),\ 
 A_{i,i+d}=a_i\I\ (a_i\in\F) \bigr\},
\]
and therefore $\gamma_d(N)>[G,\kN{d}{N}]$, as required.
\end{proof}

\section*{Addendum \today}

 After publication of this paper~\cite{KMR:SQSP}
 we were told that our Theorem~\ref{thm1} was already observed by Y.~Gorchakov~\cite[Lemma 1]{Go:LNG}.
 Moreover, for every natural number $n$, E.~Khukhro~\cite{Kh:NSP} constructed a finite group $G$, $m\in\N$ and
 subdirect products $S,N$ in $G^m$ such that $N$ is normal in $S$ and the nilpotency class $S/N$ is $n-1$.
 Later Gorchakov~\cite{Go:QGS} provided such examples with $m=n$, which is the smallest possible number of factors
 to realize such a quotient of class $n-1$. Both Khukhro's and Gorchakov's factor groups $G$ are nilpotent
 themselves in contrast to the perfect factors that we consider in this note. With this restriction we need
 subdirect products of $G^{2^{n-1}}$ to realize a quotient of class $n-1$. We do not know whether in our setting
 the number of factors can be reduced from $2^{n-1}$.


\end{document}